\documentclass{amsart}
\usepackage[utf8x]{inputenc}

\usepackage{amsthm,amsmath,amssymb,amstext,amsfonts,color}

\newcommand{\field}[1]{\mathbb{#1}}

\newcommand{\N}{\field{N}}

\numberwithin{equation}{section}
\newtheorem{theorem}{Theorem}[section]
\newtheorem{lemma}[theorem]{Lemma}
\newtheorem{corollary}[theorem]{Corollary}

\newtheorem{proposition}[theorem]{Proposition}
\theoremstyle{remark}

\renewenvironment{proof}[1][Proof]{\begin{trivlist}
\item[\hskip \labelsep {\bfseries #1:}]}{\qed\end{trivlist}}

\title{On a Rogers-Ramanujan type identity from crystal base theory}

\author{Jehanne Dousse}
\address{Institut f\"ur Mathematik, Universit\"at Z\"urich\\ Winterthurerstrasse 190, 8057 Z\"urich, Switzerland}
\email{jehanne.dousse@math.uzh.ch}

\author{Jeremy Lovejoy}
\address{\rm CNRS, Universit\'e Denis Diderot - Paris 7, Case 7014, 75205 Paris Cedex 13, France}
\email{lovejoy@math.cnrs.fr}

\thanks{The first author is supported by the Forschungskredit of the University of Zurich, grant no. FK-16-098. The authors thank the University of Zurich and the French-Swiss collaboration project no. 2015-09 for funding research visits during which this research was conducted.}

\begin{document}

\begin{abstract}
We refine and generalise a Rogers-Ramanujan type partition identity arising from crystal base theory.   Our proof uses the variant of the method of weighted words recently introduced by the first author.
\end{abstract}

\maketitle

\section{Introduction}
As combinatorial statements the \emph{Rogers-Ramanujan identities} assert that for $i = 0$ or $1$ and for all non-negative integers $n$, the number of partitions of $n$ into parts differing by at least two and having at most $i$ ones is equal to the number of partitions of $n$ into parts congruent to $\pm (2-i)$ modulo $5$.   As $q$-series identities they say that 
\begin{equation} \label{R-R}
\sum_{n \geq 0} \frac{q^{n^2+ (1-i)n}}{(q;q)_n} = \frac{1}{(q^{2-i};q^5)_{\infty}(q^{3+i};q^5)_{\infty}},
\end{equation} 
where for $n \in \mathbb{N} \cup \{\infty\}$ we have
$$
(a;q)_n := \prod_{k=0}^{n-1} (1-aq^k).
$$
A Lie-theoretic interpretation and proof of these identities were given by Lepowsky and Wilson \cite{Lepowsky,Lepowsky2}.   Up to a factor of $(-q;q)_{\infty}$, the right-hand side of \eqref{R-R} is the principally specialised Weyl-Kac character formula for level $3$ standard $A_1^{(1)}$-modules \cite{Le-Mi1,Le-Mi2}, while the product of this factor and the left-hand side corresponds to bases constructed from vertex operators.

The vertex operator approach of Lepowsky and Wilson was subsequently extended by many authors to treat level $k$ and/or other affine Lie algebras, beginning a fruitful interaction between Lie theory and partition theory.    For some examples of vertex operator constructions leading to partition identities, see \cite{Capparelli,Capparelli2,Meurman,Meurman2,Meurman3,Siladic}, and for some combinatorial approaches to such partition identities we refer to \cite{AllAndGor,Andrewscap,Doussesil,Doussesil2}.      

In \cite{Primc2}, Primc observed that the difference conditions in certain vertex operator constructions correspond to energy functions of perfect crystals, and in \cite{Primc} he studied partition identities of the Rogers-Ramanujan type coming from crystal base theory.  (For other early examples of the study of Rogers-Ramanujan type identities from the point of view of crystal bases, see \cite{Jing-Misra,OSS1,S-W1}.)  Here the Weyl-Kac character formula again gives the partitions defined by congruence conditions, while the crystal base character formula of Kang, Kashiwara, Misra, Miwa, Nakashima and Nakayashiki \cite{KMN2} ensures the correspondence with partitions defined by difference conditions.

In this paper we will be concerned with the following partition identity of Primc.   Consider partitions $(\lambda_1, \lambda_2,\dots)$ into parts of four colours $a,b,c,d$, with the order
\begin{equation}
\label{orderprimc}
1_{a} < 1_{b} < 1_{c} <1_{d} <2_{a} < 2_{b} <2_{c} < 2_{d} < \cdots ,
\end{equation}
where $k_z$ denotes the part $k$ of colour $z$ for $k \in \N$ and $z \in \{a,b,c,d\}$.
Let the minimal difference between consecutive parts of colour $x$ and $y$ be given by the entry $(x,y)$ in the matrix
\begin{equation} \label{Dmatrix}
D=\bordermatrix{\text{} & a & b & c & d \cr a & 2&1&2&2 \cr b &1&0&1&1 \cr c &0&1&0&2 \cr d&0&1&0&2}.
\end{equation}
Then 
$$
\sum_{\lambda} q^{\sum_{k \geq 1}\big( (2k-1)A_k(\lambda) + 2k (B_k(\lambda) + C_k(\lambda)) + (2k+1)D_k(\lambda)\big)} = \frac{1}{(q)_{\infty}}, 
$$
where the sum is over the coloured partitions $\lambda$ satisfying the difference conditions given by \eqref{Dmatrix} and where $A_k(\lambda)$ (resp. $B_k(\lambda), C_k(\lambda), D_k(\lambda)$) denotes the number of parts $k$ of colour $a$ (resp. $b, c ,d$) in $\lambda$. 
In other words, if the coloured integers in \eqref{orderprimc} are transformed by
\begin{equation} \label{Primcdilation}
\begin{aligned}
k_{a} &\rightarrow 2k-1,
\\k_{b} &\rightarrow 2k,
\\k_{c} &\rightarrow 2k,
\\k_{d} &\rightarrow 2k+1,
\end{aligned}
\end{equation}
the generating function for the resulting coloured partitions with the difference conditions inherited from \eqref{Dmatrix} is equal to the generating function for ordinary partitions\footnote{This was actually stated with a question mark by Primc, who was unsure of the application of the crystal base formula of \cite{KMN2} to the case of the $A_1^{(1)}$-crystal whose energy matrix is \eqref{Dmatrix}.   We are indebted to K. Misra for pointing out that this case is covered by Section 1.2 of \cite{KMN2}, rendering Primc's question mark unnecessary.}. 

Our main result is a generalisation and refinement of Primc's identity.   

\begin{theorem}
\label{th:main}
Let $A(n;k,\ell,m)$ denote the number of four-coloured partitions of $n$ with the ordering \eqref{orderprimc} and matrix of difference conditions \eqref{Dmatrix}, having $k$ parts coloured $a$, $\ell$ parts coloured $c$ and $m$ parts coloured $d$. Then
$$\sum_{n,k,\ell,m \geq 0} A(n;k,\ell,m) q^n a^k c^{\ell} d^m = \frac{(-aq;q^2)_{\infty}(-dq;q^2)_{\infty}}{(q;q)_{\infty}(cq;q^2)_{\infty}}.$$
\end{theorem}

Under the dilations
\begin{equation} \label{Primcrefineddilation}
\begin{aligned}
q &\rightarrow q^2,
\\a &\rightarrow aq^{-1},
\\d &\rightarrow dq,
\end{aligned}
\end{equation}
the integers in \eqref{orderprimc} are transformed by
\begin{equation}
\begin{aligned}
k_a &\rightarrow (2k-1)_a,
\\k_b &\rightarrow 2k_b,
\\k_c &\rightarrow 2k_c,
\\k_d &\rightarrow (2k+1)_d,
\end{aligned}
\end{equation}
their order becomes
$$1_a < 2_b <2_c < 3_d <3_a <4_b <4_c<5_d < \cdots $$
and the matrix $D$ in \eqref{Dmatrix} becomes
$$
D_2=\bordermatrix{\text{} & a & b & c & d \cr a & 4&1&3&2 \cr b &3&0&2&1 \cr c &1&2&0&3 \cr d&2&3&1&4}.
$$
Considering the $a$-parts and $c$-parts together coloured red and the $b$-parts and the $d$-parts together coloured green, this gives the following refinement of Primc's identity in terms of two-coloured partitions.

\begin{corollary}
\label{th:primcrefined}
Let $\mathcal{P}_2$ denote the set of partitions where parts may appear in two colours, say red and green, and let $c(\lambda_i)$ denote the colour of a part $\lambda_i$.
Let $A_2(n;k,\ell,m)$ denote the number of partitions $(\lambda_1, \lambda_2,\dots)$ of $n$ in $\mathcal{P}_2$ having $k$ odd red parts, $\ell$ even red parts, and $m$ odd green parts, such that no part is a green $1$ and
$$
\lambda_i - \lambda_{i+1} \geq 
\begin{cases}
1,& \text{if $\lambda_i$ is odd and $c(\lambda_i) \neq c(\lambda_{i+1})$}, \\
2,& \text{if $\lambda_i$ is even and $c(\lambda_i) \neq c(\lambda_{i+1})$}, \\
3,& \text{if $\lambda_i$ is odd and $c(\lambda_i) = c(\lambda_{i+1})$}.
\end{cases}
$$
Then 
\begin{equation} \label{fiere}
\sum_{n,k,\ell,m \geq 0} A_2(n;k,\ell,m) q^n a^k c^{\ell} d^m = \frac{(-aq;q^4)_{\infty}(-dq^3;q^4)_{\infty}}{(q^2;q^2)_{\infty}(cq^2;q^4)_{\infty}}.
\end{equation}
In other words, if $B_2(n;k,\ell,m)$ denotes the number of partitions of $n$ in $\mathcal{P}_2$ such that odd parts are distinct and only parts congruent to $2$ modulo $4$ may be green, having $k$ parts congruent to $1$ modulo $4$, $\ell$ green parts, and $m$ parts congruent to $3$ modulo $4$, then 
$$
A_2(n;k,\ell,m)  =  B_2(n;k,\ell,m).
$$ 
\end{corollary}

One recovers Primc's identity by setting $a=c=d=1$, as the dilations in \eqref{Primcrefineddilation} correspond to \eqref{Primcdilation} and the infinite product in \eqref{fiere} becomes
\begin{equation*}
\frac{(-q;q^4)_{\infty}(-q^3;q^4)_{\infty}}{(q^2;q^2)_{\infty}(q^2;q^4)_{\infty}} 
=\frac{1}{(q;q)_{\infty}}.
\end{equation*}

Another nice application of Theorem \ref{th:main} is the dilation
\begin{align*}
q &\rightarrow q^4,
\\a &\rightarrow aq^{-3},
\\c &\rightarrow cq^{-2},
\\d &\rightarrow dq^3,
\end{align*}
where the ordering of integers~\eqref{orderprimc} becomes
$$1_a < 2_c < 4_b < 5_a < 6_c < 7_d < 8_b < 9_a < \cdots ,$$
the matrix $D$ in \eqref{Dmatrix} becomes
$$
D_4=\bordermatrix{\text{} & a & b & c & d \cr a & 8&1&7&2 \cr b &7&0&6&1 \cr c &1&2&0&3 \cr d&6&7&5&8},
$$
and we obtain the following partition identity.
\begin{corollary} \label{th:q^4dilation}
Let $A_4(n;k,\ell,m)$ denote the number of partitions $\lambda = (\lambda_1,\lambda_2,\dots)$ of $n$ with $k,\ell$, and $m$ parts congruent to $1$, $2$, and $3$ modulo $4$, respectively, with no part equal to $3$, such that $\lambda_i - \lambda_{i+1} \geq 5$ if ($i$) $\lambda_ i \equiv 3 \pmod{4}$ or if ($ii$) $\lambda_i \equiv 0,1 \pmod{4}$ and $\lambda_{i+1} \equiv 1,2 \pmod{4}$.   Then
\begin{equation}
\sum_{n,k,\ell,m \geq 0} A_4(n;k,\ell,m) q^n a^k c^{\ell} d^m = \frac{(-aq;q^8)_{\infty}(-dq^7;q^8)_{\infty}}{(q^4;q^4)_{\infty}(cq^2;q^8)_{\infty}}.
\end{equation}
In other words, if $B_4(n;k,\ell,m)$ denotes the number of partitions of $n$ into even parts not congruent to $6$ modulo $8$ and distinct odd parts congruent to $\pm 1$ modulo $8$, with $k,\ell$, and $m$ parts congruent to $1,2$, and $7$ modulo $8$, respectively, then 
$$
A_4(n;k,\ell,m) = B_4(n;k,\ell,m).
$$
\end{corollary}

The proof of Theorem~\ref{th:main} relies on the variant of the method of weighted words recently introduced by the first author~\cite{Dousseunif,Doussesil2}. The difference with the original method of Alladi and Gordon~\cite{Alladi} is that instead of using the minimal partitions and $q$-series identities, we use recurrences and $q$-difference equations (with colours) coming from the difference conditions in \eqref{Dmatrix} and we solve them directly.    This is presented in the next section, and in Section 3 we give some examples and another application of Theorem \ref{th:main}.

\section{Proof of Theorem~\ref{th:main}}
\subsection{Idea of the proof}
To prove Theorem~\ref{th:main}, we proceed as follows.

Define $G_k=G_k (q;a,c,d)$ (resp. $E_k=E_k (q;a,c,d)$) to be the generating function for coloured partitions satisfying the difference conditions from \eqref{Dmatrix} with the added condition that the largest part is at most (resp. equal to) $k$.

Then we want to find $\lim_{k \rightarrow \infty} G_k$, which is the generating function for all partitions with difference conditions, as there is no more restriction on the size of the largest part.

We start by using \eqref{Dmatrix} to give simple recurrence equations relating the $G_k$'s and the $E_k$'s. 
Then we combine them to obtain a big recurrence equation involving only $G_{k_d}$'s.
This is done in Section~\ref{sec:qdiff}.

Then we use the technique consisting of going back and forth from $q$-difference equations to recurrences introduced by the first author~\cite{Doussegene,Dousseunif,Doussegene2}, and conclude using Appell's  comparison theorem. This is done in Section~\ref{sec:backandforth}.

\subsection{Recurrences and $q$-difference equations}
\label{sec:qdiff}
We use combinatorial reasoning on the largest part of partitions to state some recurrences. We have the following identities:

\begin{lemma}
For all $k \geq 1,$ we have
\label{equations}
\begin{subequations}
\begin{align}
G_{k_{d}}-G_{k_{c}}=E_{k_{d}}&=dq^k \left(E_{k_{c}}+E_{k_{a}}+G_{(k-1)_{c}}\right), \label{eq1} \\
G_{k_{c}}-G_{k_{b}}=E_{k_{c}}&=cq^k \left(E_{k_{c}}+E_{k_{a}}+G_{(k-1)_{c}}\right), \label{eq2} \\
G_{k_{b}}-G_{k_{a}}=E_{k_{b}}&=q^k \left(E_{k_{b}}+G_{(k-1)_{d}}\right), \label{eq3} \\
G_{k_{a}}-G_{(k-1)_{d}}=E_{k_{a}}&=aq^k \left(E_{(k-1)_{b}}+G_{(k-2)_{d}}\right),\label{eq4}
\end{align}
\end{subequations}
with the initial conditions
\begin{align*}
E_{0_{a}}&=E_{0_{c}}=E_{0_{d}}=0,
\\ E_{0_{b}}&=1,
\\ G_{{-1}_d}&=G_{0_{a}}=0,
\\ G_{0_{b}}&=G_{0_{c}}=G_{0_{d}}=1,
\end{align*}
\end{lemma}
\begin{proof}
We give details only for \eqref{eq1}.    The other identities follow in a similar manner.    The first equality   $G_{k_{d}}-G_{k_{c}} = E_{k_{d}}$ follows directly from the definitions.   Next, in a partition counted by $E_{k_{d}}$ we remove the largest part of size $k$ and colour $d$, giving the factor $dq^k$.   An examination of the difference conditions in \eqref{Dmatrix} shows that in the partition remaining the largest part could be $k_c$, $k_a$, or a part at most $(k-1)_c$.   This corresponds to the terms $E_{k_{c}}+E_{k_{a}}+G_{(k-1)_{c}}$.
\end{proof}

The recurrences \eqref{eq1}-\eqref{eq4} completely characterise the coloured partitions with difference conditions of Theorem~\ref{th:main}.

\medskip

Next we give a recurrence equation involving only $G_{k_{d}}$'s.

\begin{proposition}
\label{prop:qdiff}
For all $k \geq 3$ we have
\begin{equation}
\label{eq:qdiff}
\begin{aligned}
(1-cq^k)G_{k_{d}}&= \frac{1-cq^{2k}}{1-q^k}G_{(k-1)_{d}} 
\\&+ \frac{aq^k+dq^k+adq^{2k}}{1-q^{k-1}}G_{(k-2)_{d}} +\frac{adq^{2k-1}}{1-q^{k-2}}G_{(k-3)_{d}},
\end{aligned}
\end{equation}
with the initial conditions
\begin{align*}
G_{0_d} &= 1,\\
G_{1_d} &= \frac{q}{1-q} + \frac{(1+aq)(1+dq)}{1-cq},\\
G_{2_d} &= \frac{q^3}{(1-q)(1-q^2)} + \frac{(1+aq)(1+dq)(1-q^3)}{(1-q)(1-q^2)(1-cq)}.
\end{align*}
\end{proposition}

\begin{proof}
To find the correct initial conditions, we use Lemma \ref{equations}. Now let us prove~\eqref{eq:qdiff}.

We first observe that
\begin{equation}
\label{eq:GbGd}
G_{k_{b}} = G_{(k-1)_{d}} +E_{k_{a}} +E_{k_{b}}.
\end{equation}

By Equation~\eqref{eq3}, it is clear that for all $k$,
\begin{equation}
\label{eq:Eb}
E_{k_{b}}= \frac{q^k}{1-q^k} G_{(k-1)_{d}}.
\end{equation}

Now substituting this with $k$ replaced by $k-1$ into Equation~\eqref{eq4}, we get
\begin{equation}
\label{eq:Ea}
E_{k_{a}}= \frac{aq^{k}}{1-q^{k-1}} G_{(k-2)_{d}}.
\end{equation}

Thus combining Equations~\eqref{eq:GbGd},~\eqref{eq:Ea} and~\eqref{eq:Eb}, we obtain
\begin{equation}
\label{eq:Gb}
G_{k_{b}}= \frac{1}{1-q^k} G_{(k-1)_{d}}+\frac{aq^k}{1-q^{k-1}} G_{(k-2)_{d}}.
\end{equation}

Let us now turn to $E_{k_{c}}$. By Equation~\eqref{eq2}, we have
$$E_{k_{c}}= \frac{cq^k}{1-cq^k} \left( E_{k_{a}}+G_{(k-1)_{c}} \right).$$
Substituting~\eqref{eq:Ea}, we obtain
\begin{equation}
\label{eq:Ec}
E_{k_{c}}= \frac{cq^k}{1-cq^k} \left( \frac{aq^{k}}{1-q^{k-1}} G_{(k-2)_{d}}+G_{(k-1)_{c}}\right).
\end{equation}

Finally, by Equations~\eqref{eq1} and~\eqref{eq2} and the initial conditions, for all $k$, we have
$$d E_{k_{c}} = c E_{k_{d}}.$$
Combining that with~\eqref{eq:Ec}, we obtain that for all $k$,
\begin{equation}
\label{eq:Ed}
E_{k_{d}}= \frac{dq^k}{1-cq^k} \left( \frac{aq^{k}}{1-q^{k-1}} G_{(k-2)_{d}}+G_{(k-1)_{c}} \right).
\end{equation}

Using Equations~\eqref{eq:Gb},~\eqref{eq:Ec},~\eqref{eq:Ed} and the fact that
$$G_{k_{d}} = G_{k_{b}}+E_{k_{c}} +E_{k_{d}},$$
we obtain
\begin{align*}
G_{k_{d}} &= \frac{1}{1-q^k} G_{(k-1)_{d}} +\frac{aq^k}{1-q^{k-1}} G_{(k-2)_{d}} \\
&+ \frac{(c+d)q^k}{1-cq^k} \left( \frac{aq^{k}}{1-q^{k-1}} G_{(k-2)_{d}}+G_{(k-1)_{c}} \right).
\end{align*}
Rearranging gives an expression for $G_{(k-1)_{c}}$ in terms of $G_{k_{d}}$'s.
\begin{equation}
\label{eq:Gc}
G_{(k-1)_{c}} = \frac{1-cq^k}{(c+d)q^k} \left( G_{k_{d}} - \frac{1}{1-q^{k}} G_{(k-1)_{d}}- \frac{aq^{k}(1+dq^k)}{(1-q^{k-1})(1-cq^k)} G_{(k-2)_{d}}\right).
\end{equation}

Substituting this into~\eqref{eq:Ec} and simplifying leads to
\begin{equation}
\label{eq:Ec'}
E_{k_{c}} = \frac{c}{c+d} G_{k_{d}}- \frac{c}{(c+d)(1-q^{k})} G_{(k-1)_{d}}- \frac{acq^{k}}{(c+d)(1-q^{k-1})} G_{(k-2)_{d}}.
\end{equation}

On the other hand, using~\eqref{eq:Gb},~\eqref{eq:Gc} and the fact that
$$E_{k_{c}} = G_{k_{c}}- G_{k_{b}},$$
we obtain
\begin{equation}
\label{eq:Ec''}
\begin{aligned}
E_{k_{c}} &= \frac{1-cq^{k+1}}{(c+d)q^{k+1}} G_{(k+1)_{d}} - \frac{1-cq^{k+1}}{(c+d)q^{k+1}(1-q^{k+1})} G_{k_{d}} \\
&- \frac{a+c+d+adq^{k+1}}{(c+d)(1-q^{k})} G_{(k-1)_{d}}- \frac{aq^k}{1-q^{k-1}} G_{(k-2)_{d}}.
\end{aligned}
\end{equation}

Equating~\eqref{eq:Ec'} and~\eqref{eq:Ec''} and replacing $k$ by $k-1$ yields the desired recurrence equation.
\end{proof}

\subsection{Finding $ \lim_{k \rightarrow \infty} G_{k}(q;a,c,d)$}
\label{sec:backandforth}
To finish the proof we wish to calculate \\ $\lim_{k \rightarrow \infty} G_{k_d}(q;a,c,d)$, where the $G_{k_d}$'s satisfy the recurrence of order $3$ in \eqref{eq:qdiff}. This is in constrast to classical partition identities, where the recurrence/$q$-difference equation is typically of order $1$ (see for example \cite{Andrews2,Generalisation2,Andrews1,Generalisation1}). The problem of treating higher order recurrences/$q$-difference equations has recently come up in work of the first author on overpartition identities \cite{Dousse,Doussegene,Dousseunif,Doussegene2}, and her method applies here as well. Specifically, we transform the recurrence of order $3$ into a simple one of order $2$, and then, as in the classical case, we apply Appell's  comparison theorem \cite{Appell} to find the desired limit.  In the conclusion we sketch an alternative method suggested by the referee, which gives a $q$-hypergeometric generating function for $G_{k_d}(q;a,c,d)$ from which the limit also follows.

For all $k \geq 0$, let us define $$H_k := \frac{G_{k_{d}}(q)}{1-q^{k+1}}.$$
Thus $(H_k)$ satisfies the following recurrence equation for $k \geq 0$ :
\begin{equation}
\label{eq:H}
(1-cq^k-q^{k+1}+cq^{2k+1})H_k = (1-cq^{2k})H_{k-1}
+ (aq^k+dq^k+adq^{2k})H_{k-2} +adq^{2k-1}H_{k-3},
\end{equation}
To obtain the correct values of $H_k$ for all $k \geq 0$ using Equation~\eqref{eq:H}, we define the initial values $H_{-1}=1$ and $H_{k}=0$ for all $k \leq -2$.

We now define
$$f(x):= \sum_{k \geq 0} H_{k-1} x^k,$$
and convert Equation~\eqref{eq:H} into a $q$-difference equation on $f$:
\begin{equation}
\label{eq:f}
(1-x)f(x)= (1+\frac{c}{q}+ax^2q+dx^2q)f(xq)-(1+xq)(\frac{c}{q}-adx^2q^2)f(xq^2),
\end{equation}
together with the initial conditions
\begin{align*}
f(0)&= H_{-1} = 1,
\\ f'(0) &= H_{0}=\frac{1}{1-q}.
\end{align*}
This is a $q$-difference equation of order $2$, which is still not obvious to solve. But we make another transformation to obtain a very simple recurrence of order $2$. Define
$$g(x):= \frac{f(x)}{(-x;q)_{\infty}}.$$
We obtain:
\begin{equation}
\label{eq:g}
(1-x^2)g(x)= (1+\frac{c}{q}+ax^2q+dx^2q)g(xq)-(\frac{c}{q}-adx^2q^2)g(xq^2),
\end{equation}
and
\begin{align*}
g(0)&=f(0) = 1,
\\ g'(0) &= f'(0) - \frac{f(0)}{1-q}=\frac{1}{1-q} - \frac{1}{1-q}=0.
\end{align*}

Finally let us define $(a_n)$ as
$$\sum_{n \geq 0} a_n x^n := g(x).$$
Then $(a_n)$ satisfies the recurrence equation
\begin{equation*}
\left(1-q^n -cq^{n-1}+cq^{2n-1}\right)a_n = \left(1+aq^{n-1}+dq^{n-1}+adq^{2n-2}\right)a_{n-2},
\end{equation*}
which simplifies as
\begin{equation}
\label{eq:a}
a_n = \frac{\left(1+aq^{n-1}\right)\left(1+dq^{n-1}\right)}{\left(1-q^n\right)\left(1-cq^{n-1}\right)} a_{n-2},
\end{equation}
and the initial conditions
\begin{align*}
a_0&=g(0) = 1,
\\ a_1 &= g'(0)=0.
\end{align*}

Thus for all $n \geq 0$, we have
$$a_{2n} = \frac{(-aq;q^2)_n(-dq;q^2)_n}{(q^2;q^2)_n(cq;q^2)_n}a_0 = \frac{(-aq;q^2)_n(-dq;q^2)_n}{(q^2;q^2)_n(cq;q^2)_n},$$
and
$$a_{2n+1} = \frac{(-aq^2;q^2)_n(-dq^2;q^2)_n}{(q^3;q^2)_n(cq^2;q^2)_n}a_1 = 0.$$

We now conclude using Appell's  comparison theorem. We have
\begin{align*}
\lim_{k \rightarrow \infty} G_k(q;a,c,d) &= \lim_{k \rightarrow \infty} H_k\\
&= \lim_{x \rightarrow 1^-} (1-x) \sum_{k \geq 0}  H_{k-1} x^k\\
&= \lim_{x \rightarrow 1^-} (1-x) f(x)\\
&= \lim_{x \rightarrow 1^-} (1-x) g(x) \prod_{k \geq 0} (1+xq^k)\\
&= (-q;q)_{\infty} \lim_{x \rightarrow 1^-} (1-x^2) \sum_{n \geq 0} a_{2n} x^{2n}\\
&= (-q;q)_{\infty} \lim_{n \rightarrow \infty} a_{2n}\\
&= \frac{(-q;q)_{\infty}(-aq;q^2)_{\infty}(-dq;q^2)_{\infty}}{(q^2;q^2)_{\infty}(cq;q^2)_{\infty}}\\
&=  \frac{(-aq;q^2)_{\infty}(-dq;q^2)_{\infty}}{(q;q)_{\infty}(cq;q^2)_{\infty}}.
\end{align*}
We used Appell's theorem on the second line, and on the sixth with $x$ replaced by $x^2$.

\section{Examples and further results}
We begin this section by illustrating Corollaries \ref{th:primcrefined} and \ref{th:q^4dilation}.    
First, the eleven two-coloured partitions of $6$ satisfying the difference conditions in Corollary \ref{th:primcrefined} and having no green $1$ are the following, where green parts are marked with a prime:
$$
\begin{gathered}
(6), (6'), (5,1), (5',1), (4,2), (4',2), (4,2'), (4',2'), \\
(3',2,1), (2,2,2), (2',2',2').
\end{gathered}
$$
On the other hand, the eleven two-coloured partitions with distinct odd parts where only parts $2$ modulo $4$ can be green are 
$$
\begin{gathered}
(6), (6'), (5,1), (4,2), (4, 2'), (3,2,1), (3,2',1), \\
(2,2,2), (2,2,2'), (2,2',2'), (2',2',2').
\end{gathered}
$$
One may then easily verify that $A_2(6;k,\ell,m) = B_2(6;k,\ell,m)$ for a given choice of $(k,\ell,m)$.   For example, 
$A_2(6;1,0,1) = B_2(6;1,0,1) = 1$, the relevant partitions being $(5',1)$ and $(3,2,1)$, respectively. 

Next, the thirteen partitions of $14$ satisfying the difference conditions in Corollary \ref{th:q^4dilation} and having no part equal to $3$ are
$$
\begin{gathered}
(14), (13,1), (12,2), (11,2,1), (10,4), (10,2,2), (9,2,2,1),  \\
(8,2,2,2), (7,2,2,2,1), (6,6,2), (6,4,4), (6,2,2,2,2), (2,2,2,2,2,2,2),
\end{gathered}
$$
while the thirteen partitions of $14$ satisfying the congruence conditions are 
$$
\begin{gathered}
(12,2), (10,4), (10,2,2), (9,4,1), (9,2,2,1), (8,4,2), (8,2,2,2), \\
(7,4,2,1), (7,2,2,2,1), (4,4,4,2), (4,4,2,2,2), (4,2,2,2,2,2), (2,2,2,2,2,2,2).
\end{gathered}
$$
Again, one easily verifies that $A_4(13;k,\ell,m) = B_4(13;k,\ell,m)$ for a given choice of $(k,\ell,m)$.    

We close with one more application of Theorem \ref{th:main}.    Here parts divisible by $3$ may appear in two kinds. Performing the dilation
\begin{align*}
q &\rightarrow q^3,
\\a &\rightarrow aq^{-1},
\\c &\rightarrow 1,
\\d &\rightarrow dq,
\end{align*}
the ordering of integers~\eqref{orderprimc} becomes
$$2_a < 3_b < 3_c < 4_d < 5_a < 6_b < 6_c < 7_d < 8_a < 9_b < 9_c < \cdots $$
and the matrix $D$ in \eqref{Dmatrix} becomes
$$
D_3=\bordermatrix{\text{} & a & b & c & d \cr a & 6&2&5&4 \cr b &4&0&3&2 \cr c &1&3&0&5 \cr d&2&4&1&6}.
$$
Letting $b$-parts and $c$-parts be ordinary and primed multiples of $3$, respectively, we obtain the following partition identity.
\begin{corollary}
Let $\mathcal{P}_3$ denote the set of partitions where parts divisible by $3$ may appear in two kinds, say ordinary and primed.  Let $A_3(n;k,m)$ denote the number of partitions of $n$ in $\mathcal{P}_3$ with $k$ and $m$ parts congruent to $2$ and $1$ modulo $3$, respectively, such that $\lambda_i \neq 1$ and 
$$
\lambda_i - \lambda_{i+1} \geq 
\begin{cases}
3, &\text{if $(\lambda_i,\lambda_{i+1}) \pmod{3} \subset (\{0,2\},\{0',2\})$ or  $(\{0',1\},\{0,1\})$}, \\
4, &\text{if $3 \nmid \lambda_i, \lambda_{i+1}$ and  $\lambda_i - \lambda_{i+1} \not \equiv 2 \pmod{3}$}.
\end{cases}
$$
Then 
\begin{equation} \label{almostCapp}
\sum_{n,k,m \geq 0} A_3(n;k,m)q^n a^k d^{m} = \frac{(-aq^2;q^6)_{\infty} (-dq^4;q^6)_{\infty} (-q^3;q^3)_{\infty}}{(q^3;q^3)_{\infty}}.
\end{equation}

In other words, if $B_3(n;k,m)$ denotes the number of partitions of $n$ in $\mathcal{P}_3$ with $k$ and $m$ parts congruent to $2$ and $4$ modulo $6$, respectively, such that primed multiples of $3$ may not repeat,  then 
$$
A_3(n;k,m) = B_3(n;k,m).
$$    
\end{corollary} 
Note that the generating function in \eqref{almostCapp}  differs only slightly from the infinite product appearing in the Alladi-Andrews-Gordon generalisation of Capparelli's identity \cite{AllAndGor},
$$
(-aq^2;q^6)_{\infty} (-bq^4;q^6)_{\infty} (-q^3;q^3)_{\infty}.
$$

\section{Concluding Remarks}
The referee has kindly pointed out that the recurrence in Proposition \ref{prop:qdiff} can be used to give a $q$-hypergeometric generating function for $G_{k_d}$,
\begin{equation} \label{Gkdgenfun}
G_{k_d} = (1-q^{k+1}) \sum_{i=0}^{\lfloor (k+1)/2 \rfloor} \frac{q^{\binom{k-2i+1}{2}}(-aq;q^2)_i(-dq;q^2)_i}{(q;q)_{k-2i+1}(q^2;q^2)_i(cq;q^2)_i}.
\end{equation}
The idea is to recursively define sequences $g_k^{(i)}$ and $h_k^{(i)}$ with $g_k^{(0)} = G_{k_d}$ by 
\begin{align*}
h_k^{(i)} &:= \lim_{c \to \infty} g_k^{(i)}, \\ 
g_{k}^{(i+1)} &:= (1-cq^{2i+1}) (g_k^{(i)} - h_k^{(i)}),
\end{align*}
where the existence of $h_k^{(i)}$ follows from the recurrence (plus initial conditions) for $g_k^{(i)}$.
At each step one uses the recurrence for the $g_k^{(i)}$ and formula for the $h_k^{(i)}$ to find a recurrence for the $g_k^{(i+1)}$ and a formula for the $h_k^{(i+1)}$. 
The recurrence for $g_k^{(i)}$ is the following:

\begin{align*}
\left( 1-cq^k \right) g_k^{(i)} &= \frac{1-cq^{2k}}{1-q^k} g_{k-1}^{(i)} + \frac{aq^k+dq^k+adq^{2k}}{1-q^{k-1}} g_{k-2}^{(i)}\\
&+ \frac{adq^{2k-1}}{1-q^{k-2}} g_{k-3}^{(i)} +  \frac{q^{\binom{k-2i+2}{2}}\left(1-cq^{2i-1} \right)(-aq;q^2)_i(-dq;q^2)_i}{(q;q)_{k-2i+1}(q^2;q^2)_{i-1}},
\end{align*}
where $1/(q;q)_n =0$ for $n <0$.    From this one can deduce the recurrence for $h_k^{(i)}$ :

$$h_k^{(i)} =  \frac{q^k}{1-q^k}h_{k-1}^{(i)} +  \frac{q^{\binom{k-2i+1}{2}}(-aq;q^2)_i(-dq;q^2)_i}{(q;q)_{k-2i+1}(q^2;q^2)_{i-1}}.$$
The result is
$$
h_k^{(i)} = (1-q^{k+1}) \frac{q^{\binom{k-2i+1}{2}}(-aq;q^2)_i(-dq;q^2)_i}{(q;q)_{k-2i+1}(q^2;q^2)_i}.
$$
We leave the details to the interested reader.   Since 
$$
G_{k_d} = g_k^{(0)} = \sum_{i \geq 0} \frac{h_k^{(i)}}{(cq;q^2)_i},
$$ 
we obtain \eqref{Gkdgenfun}.     Note that if we replace $k$ by $2k-1+\delta$ for $\delta = 0,1$ and use the fact that
$$
\sum_{i=0}^{\infty} \frac{q^{\binom{2i+\delta}{2}}}{(q;q)_{2i+\delta}} = (-q;q)_{\infty},
$$
we have
\begin{align*}
\lim_{k \to \infty} g_{2k-1+\delta}^{(0)} &= \lim_{k \to \infty} (1-q^{2k+\delta}) \sum_{i=0}^{k} \frac{q^{\binom{2i+\delta}{2}}(-aq;q^2)_{k-i}(-dq;q^2)_{k-i}}{(q;q)_{2i+\delta}(q^2;q^2)_{k-i}(cq;q^2)_{k-i}} \\
&= \frac{(-aq;q^2)_{\infty}(-dq;q^2)_{\infty}}{(q^2;q^2)_{\infty}(cq;q^2)_{\infty}} \sum_{i=0}^{\infty} \frac{q^{\binom{2i+\delta}{2}}}{(q;q)_{2i+\delta}} \\
&= \frac{(-aq;q^2)_{\infty}(-dq;q^2)_{\infty}}{(q;q)_{\infty}(cq;q^2)_{\infty}},
\end{align*}
in agreement with Section 2.3.   This kind of argument differs from the usual approach and should be kept in mind for future studies of  partition identities.

\section*{Acknowledgements}
We thank the referee for a careful reading of the paper and the many valuable suggestions for its improvement, especially for clarifying the early work on vertex operators and crystal bases and for detailing an alternative approach to Section 2.3.  



\bibliographystyle{siam}
\bibliography{biblio}

\end{document}